\documentclass[letterpaper,12pt,reqno]{amsart}
\usepackage{graphicx,amsmath,amssymb}

\usepackage{srcltx}
\usepackage[latin1]{inputenc}
\usepackage[T1]{fontenc}

\usepackage{fullpage}

\newtheorem{X}{X}[section]
\newtheorem{corollary}[X]{Corollary}

\newtheorem{lemma}[X]{Lemma}
\newtheorem{proposition}[X]{Proposition}
\newtheorem{theorem}[X]{Theorem}

\theoremstyle{definition}
\newtheorem{remark}[X]{Remark}

\title{On a theorem of Bombieri, Friedlander and Iwaniec}
\author{Daniel Fiorilli}
\address{D\'epartement de math\'ematiques et de statistique \\ Universit\'e de Montr\'eal \\ CP 6128, succ.\ Centre-ville \\ Montr\'eal, QC \\ Canada H3C 3J7}
\email{fiorilli@dms.umontreal.ca}

\begin{document}

\begin{abstract}
 In this article, we show to which extent one can improve a theorem of Bombieri, Friedlander and Iwaniec by using Hooley's variant of the divisor switching technique. We also give an application of the theorem in question, which is a Bombieri-Vinogradov type theorem for the Tichmarsh divisor problem in arithmetic progressions.
\end{abstract}
\maketitle

\section{Introduction}

The Bombieri-Vinogradov theorem implies that on average over $q\leq x^{1/2-o(1)}$, the primes less than $x$ are equidistributed in the residue classes $a \bmod q$, with $(a,q)=1$. Specifically, we have for any $A>0$ that
\begin{equation} \label{BV} \sum_{q\leq Q } \max_{a:(a,q)=1} \left| \psi(x;q,a)-\frac x{\phi(q)} \right|  \ll \frac{x}{(\log x)^{A}}, \end{equation}
where $Q=x^{1/2}/(\log x)^{A+5}$. One could ask if \eqref{BV} still holds if we take $Q = x^{\theta}$, with $\theta >\frac 12$. This would be a major achievement, since it would imply bounded gaps between primes \cite{GPY}, that is 
$$ \liminf_{n} (p_{n+1}-p_n) < \infty.$$
The Elliot-Halberstam conjecture stipulates that we can take $\theta$ to be any real number less than $1$. This conjecture is however very far from reach.

One way to get past the barrier of $Q=x^{1/2-o(1)}$ is to relax the condition on $a$. Indeed, in concrete problems, one often only needs the bound \eqref{BV} for a fixed value of $a$. Sometimes, even the absolute values are not necessary. These variants were studied very closely in a series of groundbreaking articles by Fouvry \& Iwaniec (\cite{fouvry iwaniec BV}, \cite{fouvry iwaniec PAP}), Fouvry (\cite{fouvry BV}, \cite{fouvry BV2}, \cite{fouvry titchmarsh}), and Bombieri, Friedlander \& Iwaniec (\cite{BFI}, \cite{BFI2}, \cite{BFI3}). We will list the results of these authors by increasing order of uniformity.

By fixing $a$, one can go up to $Q=x^{\frac 12+ \frac 1{(\log\log x)^B}}$.
\begin{theorem}[Bombieri, Friedlander and Iwaniec \cite{BFI2}]
Let $a\neq 0$, $x\geq y\geq 3$, and $Q^2\leq xy$. We then have

$$\sum_{\substack{Q\leq q<2Q \\ (q,a)=1}} \left| \psi(x;q,a)-\frac x{\phi(q)} \right|  \ll x \left( \frac{\log y}{\log x}\right)^2 (\log\log x)^B. $$

\end{theorem}
The best known result was obtained shortly afterwards by the same authors, and shows that one can go up to $Q=x^{\frac 12+o(1)}$, whatever the nature of the $o(1)$ is.
\begin{theorem}[Bombieri, Friedlander, Iwaniec \cite{BFI3}]
Let $a\neq 0$ be an integer and $A>0$, $2\leq Q\leq x^{3/4}$ be reals. Let $\mathcal Q$ be the set of all integers $q$, prime to $a$, from an interval $Q'<q\leq Q$. Then
\begin{multline*} \sum_{q\in \mathcal Q} \left| \pi(x;q,a)-\frac {\pi(x)}{\phi(q)} \right|  \\
\leq \left\{K\left(\theta-\frac 12\right)^{2}\frac x{\log x} +O_A\left( \frac{x (\log\log x)^2}{(\log x)^3}\right)\right \} \sum_{q\in \mathcal Q}\frac 1{\phi(q)} 
+ O_{a,A}\left( \frac x{(\log x)^A} \right),
\end{multline*}
where $\theta:=\frac{\log Q}{\log x}$ and $K$ is absolute.
\end{theorem}

Replacing the absolute values by a certain weight (see \cite{BFI} for the definition of "well factorable"), we can take $Q=x^{4/7-\epsilon}$.
\begin{theorem}[Bombieri, Friedlander and Iwaniec \cite{BFI}]
\label{BFI 4/7}
Let $a\neq 0$, $\epsilon>0$ and $Q=x^{4/7-\epsilon}$. For any well factorable function $\lambda(q)$ of level $Q$ and any $A>0$ we have 
\begin{equation}\sum_{\substack{(q,a)=1}} \lambda(q) \left( \psi(x;q,a)-\frac x{\phi(q)} \right)  \ll \frac x{(\log x)^A}. \label{BFI 4/7 equation} \end{equation}
\end{theorem}
Theorem \ref{BFI 4/7} is an improvement of a result of Fouvry \& Iwaniec \cite{fouvry iwaniec PAP}, which showed that \eqref{BFI 4/7 equation} holds with $\lambda(q)$ of level $Q=x^{9/17-\epsilon}$. 

If we remove the weight $\lambda(q)$, we can take $Q=x/(\log x)^B$, which is even further than in the Elliot-Halberstam conjecture. This result was obtained  independently by Fouvry \cite{fouvry titchmarsh} and Bombieri, Friedlander \& Iwaniec \cite{BFI} (in stronger form).
\begin{theorem}[Bombieri, Friedlander and Iwaniec \cite{BFI}]
\label{BFI}
Let $a\neq 0$, $\lambda<\frac 1{10}$ and $R<x^{\lambda}$. For any $A>0$ there exists $B=B(A)$ such that provided $QR<x/(\log x)^B$ we have 
\begin{equation}
\sum_{\substack{r\leq R \\ (r,a)=1}} \left |\sum_{\substack{q\leq Q \\ (q,a)=1}} \left( \psi(x;qr,a)-\Lambda(a)-\frac x{\phi(qr)} \right) \right| \ll_{a,A,\lambda} \frac x{(\log x)^A}.
\label{equation BFI}
\end{equation}
\end{theorem}
\begin{remark}
We subtracted $\Lambda(a)$ from $\psi(x;qr,a)$ in \eqref{equation BFI} because the arithmetic progression $a \bmod qr$ contains the prime power $p^e$ for all values of $qr$ if $a=p^e$. This induces a negligible error term in \eqref{equation BFI} (for $B>A$). 
\end{remark}

In this article we focus on Theorem \ref{BFI}. We show in Corollary \ref{corollaire limitations} that for any $A>0$,
\begin{itemize}
\item If $a=\pm1$, then Theorem \ref{BFI} holds if $B(A)>A$, and is false if $B(A)=A$.

\item If $a=\pm p^e$, then Theorem \ref{BFI} holds if $B(A)=A$, and is false if $B(A)<A$.

\item If $a$ has more than two prime factors, then Theorem \ref{BFI} holds if $B(A)>\frac{538}{743} A$.
\end{itemize}
%

One of the applications of Theorem \ref{BFI} and of Fouvry's result \cite{fouvry titchmarsh} is the best known estimate for the Titchmarsh divisor problem. We will show that Theorem \ref{BFI} yields a generalization of this result, that is a Bombieri-Vinogradov type result for the Titchmarsh divisor problem in arithmetic progressions, up to level $Q=x^{1/10-\epsilon}$. 

\section{Acknowledgements}

I would like to thank my supervisor Andrew Granville for his advice, as well as my colleagues Farzad Aryan, Mohammad Bardestani, Dimitri Dias, Tristan Freiberg and Kevin Henriot for many fruitful conversations. I would also like to thank Adam T. Felix for his comments. Ce travail a été rendu possible grâce à des bourses doctorales du Conseil de Recherche en Sciences Naturelles et en Génie du Canada et de la Faculté des Études Supérieures et Postdoctorales de l'Université de Montréal.

\section{Statement of results}
For an integer $r\geq 1$, we will use the notation $$r':=\prod_{p\mid r} p.$$ Here is our main result.
\begin{theorem}
 \label{main result}
Fix an integer $a\neq 0$ and two positive real numbers $\lambda<\frac 1{10}$ and $A$. We have for $R=R(x)\leq x^{\lambda}$ and $M=M(x)\leq (\log x)^A$ that
$$  \sum_{\substack{ \frac R2 < r \leq  R \\ (r,a)=1}} \left| \sum_{\substack{q\leq \frac x{rM} \\ (q,a)=1}} \left( \psi(x;qr,a)-\Lambda(a)-\frac x{\phi(qr)}\right) -\frac{\phi(a)}{a}\frac {x}{rM} \mu(a,r,M)\right| \ll_{a,A,\epsilon,\lambda} \frac{x}{M^{\frac{743}{538}-\epsilon}}, $$
where the "average" is given by 
$$ \mu(a,r,M):=\begin{cases}
                -\frac 12 \log M-C_5(r) &\text{ if } a=\pm 1 \\
		-\frac 12 \log p &\text{ if } a=\pm p^e \\
		0 &\text{ otherwise},
               \end{cases}
$$
with 
$$C_5(r):=\frac 12\left( \log 2\pi +1+ \gamma + \sum_p \frac{\log p}{p(p-1)}+\sum_{p\mid r} \frac{\log p}{p}\right).$$

We also have the following similar result:
\begin{equation*}  \sum_{\substack{r \leq R \\ (r,a)=1}} \left| \sum_{\substack{q\leq \frac x{RM} \\ (q,a)=1}} \left( \psi(x;qr,a)-\Lambda(a)-\frac x{\phi(qr)}\right) -\frac{\phi(a)}{a}\frac {x}{RM} \mu(a,r,RM/r)\right|  \ll_{a,A,\epsilon,\lambda}  \frac{x}{M^{\frac{743}{538}-\epsilon}}. 
\end{equation*}

\end{theorem}

As a corollary, we get a more precise form of Theorem \ref{BFI}.
\begin{corollary}
\label{corollaire limitations}
Fix an integer $a\neq 0$ and two positive real numbers $\lambda<\frac 1{10}$ and $A$. We have for $R=R(x)\leq x^{\lambda}$ and $M=M(x)\leq (\log x)^A$ that
$$ \sum_{\substack{r\leq R \\ (r,a)=1}} \left |\sum_{\substack{q\leq \frac x{RM} \\ (q,a)=1}} \left( \psi(x;qr,a)-\Lambda(a)-\frac x{\phi(qr)} \right) \right| = \left(\frac{\phi(a)}{a}\right)^2 \frac x{M}\nu(a,M)+O_{a,A,\epsilon,\lambda}\left(\frac{x}{M^{\frac{743}{538}-\epsilon}}\right), $$
where 
$$ \nu(a,M):=\begin{cases}
                \frac 12 \log M+C_6+O\left(\frac{\log(RM)}R\right) & \text{ if } a=\pm 1 \\
		  \frac 12 \log p+O\left(\frac 1R\right) & \text{ if } a=\pm p^e \\
		0 &\text{ otherwise,}
             \end{cases}
$$
with 
$$ C_6:= C_5(1) + \frac 12+\frac 12\sum_p \frac{\log p}{p^2}. $$
%
%

%
%
%
%
\end{corollary}
\begin{remark}
If $a$ has at most $1$ prime factor, then for $M$ and $R$ both tending to infinity we have that
$$ \nu(a,M)\sim \begin{cases}
               \frac 12 \log M & \text{ if } a=\pm 1 \\
		 \frac 12 \log p & \text{ if } a=\pm p^e. \\
             \end{cases}
$$
(If $R$ is bounded, then we should multiply by $ \frac{a}{\phi(a)} \frac{\# \{ r\leq R : (r,a)=1 \}}{R} $ in the case $a=\pm p^e$, and by $\frac{\lfloor R \rfloor}{R}$ in the case $a=\pm 1$.)
\end{remark}

Another corollary of our results (which actually follows from Theorem \ref{BFI}) is a Bombieri-Vinogradov type result for the Titchmarsh divisor problem in arithmetic progressions. We use the following notation for the divisor function: $\tau(n):=\sum_{d\mid n}1$.
\begin{theorem}
\label{B-V pour titchmarsh thm}
Fix an integer $a\neq 0$ and let $\lambda<\frac 1{10}$ and $A$ be two fixed positive real numbers. We have for $Q\leq x^{\lambda}$ that
\begin{equation}
\label{B-V pour titchmarsh}
\sum_{q\leq Q} \left| \sum_{|a|/q<m\leq x/q} \Lambda(qm+a)\tau(m) - M.T.\right| \ll_{a,A,\lambda} \frac{x}{(\log x)^A}, 
\end{equation}
where the main term is
$$ M.T. := \frac xq \left( C_1(a,q)\log x +  2C_2(a,q) +C_1(a,q)\log\left(\frac{(q')^2}{eq}\right) \right),$$
with $C_1(a,q)$ and $C_2(a,q)$ defined as in section \ref{section notation}.
\end{theorem}

A version of Theorem \ref{B-V pour titchmarsh thm} was obtained independently by Felix \cite{felix}, who also showed how to apply this result to a question related to Artin's primitive root conjecture. 
Using Theorem \ref{B-V pour titchmarsh thm}, one can give a slight improvement of Theorem 1.5 of \cite{felix}, that is replace $O(\log\log x)$ by $c \log\log x + O(1)$, for some constant $c$.


Taking $Q=(\log x)^C$ in Theorem \ref{B-V pour titchmarsh thm}, we obtain a "Siegel-Walfisz theorem" for the Titchmarsh divisor problem, and one could ask if this is sufficient to give the bound \eqref{B-V pour titchmarsh} for $Q=x^{1/2}/(\log x)^B$, since it is known that the Bombieri-Vinogradov theorem holds with fairly general sequences satisfying the Siegel-Walfisz condition. If this is true, then it would yield the following improvement of a dyadic version of Theorem \ref{BFI}.
\begin{proposition}
\label{proposition appliquer titchmarsh a BFI}
Fix an integer $a\neq 0$, a real number $A>0$ and let $R=R(x)\leq x^{1/2}/(\log x)^{3A+5}$. Assume that \eqref{B-V pour titchmarsh} holds for $Q=R(x)$. Then for $L:=(\log x)^{A+3}$ we have
\begin{equation}
\sum_{\substack{\frac R2<r\leq R \\ (r,a)=1}} \left |\sum_{\substack{q\leq \frac x{RL} \\ (q,a)=1}} \left( \psi(x;qr,a)-\Lambda(a)-\frac x{\phi(qr)} \right) \right| \ll_{a,A}  \frac{x}{(\log x)^A}.
\end{equation}
\end{proposition}


\section{Notation}
\label{section notation}

We will denote by $\gamma$ the Euler-Mascheroni constant. We also define the following constants:
$$C_1(a,r):=\frac{ \zeta(2)\zeta(3)}{\zeta(6)}\prod_{p\mid a} \left( 1-\frac p {p^2-p+1}\right) \prod_{p\mid r}\left(1+\frac{p-1}{p^2-p+1}\right),$$
$$C_2(a,r):=C_1(a,r)\left(  \gamma - \sum_{p} \frac {\log p} {p^2-p+1}+\sum_{p\mid a} \frac{p^2\log p}{(p-1)(p^2-p+1)}-\sum_{p\mid r} \frac{(p-1)p\log p}{p^2-p+1}\right),$$
$$C_3(a,r):=C_2(a,r)-C_1(a,r),$$
$$C_5(r):=\frac 12\left( \log 2\pi + 1+\gamma + \sum_p \frac{\log p}{p(p-1)} +\sum_{p\mid r} \frac{\log p}{p}\right).$$

Moreover, for $i=1,2,3$, $$C_i(a):=C_i(a,1),$$
and 
$$ C_5:=C_5(1).$$

We denote by $\omega(n)$ the number of prime factors of $n$.

\section{Preliminary lemmas}

We start with some elementary estimates.

\begin{lemma}
\label{additif+multiplicatif}
 Let $f$ be a multiplicative function and $g$ an additive function, that is for $(m,n)=1$, $f(mn)=f(m)f(n)$ and $g(mn)=g(m)+g(n)$ (in particular, $f(1)=1$ and $g(1)=0$). Then for a squarefree integer $r$ we have that
$$ \sum_{d\mid r} f(d)g(d) = \prod_{p'\mid r} (1+f(p')) \sum_{p\mid r} \frac{g(p)f(p)}{1+f(p)}.$$
\end{lemma}
\begin{proof}
We write
\begin{align*}
\sum_{d\mid r} f(d)g(d) &= \sum_{d\mid r} f(d) \sum_{p\mid r} g(p) =  \sum_{p\mid r} g(p) \sum_{\substack{d\mid r : \\ p\mid d}} f(d) = \sum_{p\mid r} g(p) \sum_{\substack{d\mid \frac r{p}}} f(p) f(d) \\
&= \sum_{p\mid r} g(p)f(p) \prod_{p'\mid \frac r{p}} (1+f(p')) =  \sum_{p\mid r} \frac{g(p)f(p)}{1+f(p)} \prod_{p'\mid r} (1+f(p')).
\end{align*}
\end{proof}

\begin{lemma}
\label{sommes de produits singuliers}
Let $a$ and $r$ be coprime integers, with $r$ squarefree. We have for $i=1,2$ that
\begin{equation}
 \frac{C_i(a,r)}r=\sum_{d\mid r} \mu(d)C_i(ad).
\end{equation}
\end{lemma}
\begin{proof}
By the definition of $C_1(a)$, we have
$$ \sum_{d\mid r} \mu(d)C_1(ad)= C_1(a) \prod_{p\mid r} \left(1-\left(1-\frac p{p^2-p+1}\right) \right) = \frac{C_1(a,r)}r.$$

Moreover, by defining the multiplicative function $f(d):=\frac{\zeta(6)}{\zeta(2)\zeta(3)}\mu(d)C_1(d)$ we have
\begin{align*}
  \sum_{d\mid r} \mu(d)C_2(ad) &= C_1(a) \sum_{d\mid r} f(d) \left(\gamma - \sum_{p} \frac {\log p} {p^2-p+1}+\sum_{p\mid a} \frac{p^2\log p}{(p-1)(p^2-p+1)}\right) \\ &\hspace{2cm}+ C_1(a) \sum_{d\mid r} f(d)\sum_{p\mid d} \frac{p^2\log p}{(p-1)(p^2-p+1)}
\\ &= C_2(a) \prod_{p\mid r}\frac p {p^2-p+1} +C_1(a) \sum_{d\mid r} f(d)\sum_{p\mid d} \frac{p^2\log p}{(p-1)(p^2-p+1)}.
\end{align*}
Applying Lemma \ref{additif+multiplicatif}, we get that this is
\begin{align*}
&= C_2(a) \prod_{p\mid r}\frac p {p^2-p+1} +C_1(a) \prod_{p'\mid r} (1+f(p')) \sum_{p\mid r} \frac{p^2\log p}{(p-1)(p^2-p+1)} \frac{f(p)}{1+f(p)}
\\ &= C_2(a) \prod_{p\mid r}\frac p {p^2-p+1} - C_1(a) \prod_{p'\mid r} \frac {p'}{(p')^2-p'+1} \sum_{p\mid r} \frac{(p-1)p\log p}{p^2-p+1}
\\ &= C_1(a) \prod_{p\mid r}\frac p {p^2-p+1} \left(  \gamma - \sum_{p} \frac {\log p} {p^2-p+1}+\sum_{p\mid a} \frac{p^2\log p}{(p-1)(p^2-p+1)}-\sum_{p\mid r} \frac{(p-1)p\log p}{p^2-p+1}\right) \\
&=\frac{C_2(a,r)}r.
\end{align*}

\end{proof}

\begin{lemma}
 \label{estimation élémentaires}
Fix $r>0$ and $a\neq 0$ two coprime integers. We have
\begin{align*}
 \sum_{\substack{n\leq M \\ (n,a)=1}} \frac n{\phi(n)} &= C_1(a) M + O(2^{\omega(a)}\log M),  \\
\sum_{\substack{n\leq M\\ (n,a)=1}} \frac1{\phi(n)} &= C_1(a) \log M + C_2(a) + O\left( 2^{\omega(a)}\frac{\log M}{M}\right), \\
 \sum_{\substack{n\leq M\\ (n,a)=1}} \frac {rn}{\phi( rn)} &= C_1(a,r) M + O\left( 3^{\omega(ar)}\log (r'M)\right),  \\
\sum_{\substack{n\leq M \\ (n,a)=1}} \frac 1{\phi(rn)}&=\frac {C_1(a,r)} {r}\log (r'M) +\frac {C_2(a,r)} {r}+O\left(  3^{\omega(ar)}\frac{\log (r'M)}{rM} \right).
\end{align*}
\end{lemma}

\begin{proof}
For the first two estimates, see \cite{fiorilli} or \cite{FGHM}. We now sketch a proof the last estimate. First we assume that $r$ is squarefree, since if it is not we can write
$$  \frac 1{\phi(rn)}=\frac {r'}{r\phi(r'n)}. $$
Then, we use the identity
$$ \sum_{\substack{d\mid r \\ (d,n)=1}}\mu(d) = \begin{cases}
                                                 1 &\text{ if } r\mid n \\
					0& \text{ else}
                                                \end{cases}
$$
to write  
$$\sum_{\substack{n\leq M \\ (n,a)=1}} \frac 1{\phi(rn)}= \sum_{d\mid r}\mu(d) \sum_{\substack{n\leq rM \\ (n,ad)=1}}\frac 1{\phi(n)}.$$
Now, substituting in the $r=1$ estimate, we get that 
$$\sum_{\substack{n\leq M \\ (n,a)=1}} \frac 1{\phi(rn)}  =\log (rM)\sum_{d\mid r} \mu(d)C_1(ad) +\sum_{d\mid r} \mu(d) C_2(ad)+O\left(  3^{\omega(ar)}\frac{\log (rM)}{rM} \right). $$
The result follows by Lemma \ref{sommes de produits singuliers}.

\end{proof}

\begin{lemma}
 \label{somme inverse euler avec poids}
Fix $r>0$ and $a\neq 0$ two coprime integers.

If $\omega(a)\geq 1$,
\begin{equation*}
\sum_{\substack{n\leq M\\ (n,a)=1}} \frac1{\phi(nr)}\left( 1-\frac nM\right) =  \frac{C_1(a,r)}r \log (r'M) +  \frac{C_3(a,r)}r +\frac{\phi(a)}{a}\frac {\Lambda(a)}{2rM} +E(a,r,M).
\end{equation*}

If $a=\pm 1$,
\begin{equation*}
 \sum_{\substack{n\leq M\\ (n,a)=1}} \frac1{\phi(nr)}\left( 1-\frac nM\right) =  \frac{C_1(1,r)}r \log (r')M +  \frac{C_3(1,r)}r +\frac{\log (r'M)}{2rM}  + \frac{C_5}{rM}+E(a,r,M).
\end{equation*}
The error term satisfies
$$  E(a,r,M)\ll \frac{\prod_{p\mid ar} \left( 1+\frac 1{p^{\delta}}\right)}{rM} \left(\frac{a'} M \right)^{\frac{205}{538}-\epsilon},$$
for some $\delta>0$.
\end{lemma}

\begin{proof}
For the proof in the case $r=1$, we refer the reader to Lemma 6.9 of \cite{fiorilli}. In the proof, we replace $(40)$ by the bound 
$$\mathfrak S_{a_M}(s+1) \ll a_M^{-1-\sigma} \prod_{p\mid a_M} \left(1+\frac 1{p^{\delta}} \right),$$
which will yield the improved error term
$$  E(a,1,M)\ll \frac{\prod_{p\mid a} \left( 1+\frac 1{p^{\delta}}\right)}{M} \left(\frac{a'} M \right)^{\frac{205}{538}-\epsilon}.$$
Note that the exponent $\frac{205}{538}$ comes from Huxley's subconvexity bound on $\zeta(s)$ \cite{huxley}.

For the general case, we proceed as in the preceding lemma. We can again assume that $r$ is squarefree, and write
$$ \sum_{\substack{n\leq M\\ (n,a)=1}} \frac1{\phi(nr)}\left( 1-\frac nM\right) = \sum_{d\mid r} \mu(d)  \sum_{\substack{n\leq rM \\ (n,ad)=1}} \frac 1{\phi(n)} \left(1-\frac n{rM} \right), $$
in which we substitute the $r=1$ estimate. If $\omega(a)\geq 2$, then $\omega(ad)\geq 2$ for all $d\mid r$, so we get
 \begin{align*}
\sum_{\substack{n\leq M \\ (n,a)=1}} \frac 1{\phi(rn)}\left( 1-\frac nM\right) 
&=\sum_{d\mid r} \mu(d) \left( C_1(ad) \log(rM) + C_3(ad) + E(ad,1,rM)\right)
\\ &= C_1(a,r) \log(rM) + C_3(a,r) + E(a,r,M)
\end{align*}
by Lemma \ref{sommes de produits singuliers}. Here,
\begin{align*} E(a,r,M)  &\ll \sum_{d\mid r }\frac{\prod_{p\mid ad} \left( 1+\frac 1{p^{\delta}}\right)}{rM} \left(\frac{a'd} {rM} \right)^{\frac{205}{538}-\epsilon}  \\
&= \frac{\prod_{p\mid a} \left( 1+\frac 1{p^{\delta}}\right)}{rM} \left(\frac{a'} {rM} \right)^{\frac{205}{538}-\epsilon} \sum_{d\mid r} d^{\frac{205}{538}-\epsilon} \prod_{p\mid d} \left( 1+\frac 1{p^{\delta}}\right) \\
&= \frac{\prod_{p\mid a} \left( 1+\frac 1{p^{\delta}}\right)}{rM} \left(\frac{a'} {rM} \right)^{\frac{205}{538}-\epsilon}  \prod_{p\mid r} \left( 1+p^{\frac{205}{538}-\epsilon}\left(1+\frac 1{p^{\delta}} \right)\right) \\
&\ll \frac{\prod_{p\mid ar} \left( 1+\frac 1{p^{\delta}}\right)}{rM} \left(\frac{a'} {M} \right)^{\frac{205}{538}-\epsilon},
\end{align*}
where we might have to change the value of $\delta>0$.

If $\omega(a)=1$, then $\omega(ad)\geq 1$ for all $d\mid r$, so we get
 \begin{align*}
\sum_{\substack{n\leq M \\ (n,a)=1}} \frac 1{\phi(rn)}\left( 1-\frac nM\right) 
&=\sum_{d\mid r} \mu(d) \left( C_1(ad) \log(rM) + C_3(ad) + \frac{\phi(ad)}{ad}\frac {\Lambda(ad)}{2rM}+E(ad,1,rM)\right)
\\ &=\sum_{d\mid r} \mu(d) \left( C_1(ad) \log(rM) + C_3(ad)\right) +  \frac{\phi(a)}{a}\frac {\Lambda(a)}{2rM}+ E(a,r,M)
\\&= C_1(a,r) \log(rM) + C_3(a,r)+  \frac{\phi(a)}{a}\frac {\Lambda(a)}{2rM}+ E(a,r,M).
\end{align*}

If $a=\pm 1$, then we get

 \begin{align*}
\sum_{\substack{n\leq M \\ (n,a)=1}} \frac 1{\phi(rn)}\left( 1-\frac nM\right) & =\sum_{\substack{d\mid r}} \mu(d) ( C_1(ad) \log(rM) + C_3(ad)+E(ad,1,rM)) \\
& \hspace{1cm} -\sum_{p\mid r} \frac{\phi(p)}{p}\frac{\Lambda(p)}{2rM} +  \frac{\log (rM)}{2rM}  + \frac{C_5}{rM} \\
&= C_1(a,r)\log (rM)  + C_2(a,r)+\frac{\log M}{2rM}  + \frac{C_5(r)}{rM}+E(a,r,M).
\end{align*}

\end{proof}

\section{Further results and proofs}


\begin{proposition}
 \label{proposition M entier}
Fix two positive real numbers $\lambda<\frac 1{10}$ and $D$. let $M=M(r,x)$ be an integer such that $1\leq M(r,x) \leq(\log x)^D$. Then for $R=R(x)\leq x^{\lambda}$ we have
\begin{multline}
\sum_{\substack{R/2<r\leq R \\ (r,a)=1}}\Bigg|\sum_{\substack{q\leq \frac x {rM}  \\(q,a)=1}} \left( \psi(x;qr,a)-\Lambda(a)-\frac{x}{\phi(qr)} \right) \\- x\Bigg( \frac{C_1(a,r)}r\log (r'M)+\frac{C_3(a,r)}r -\sum_{\substack{s\leq M\\(s,a)=1}} \frac 1 {\phi(rs)} \left( 1-\frac sM\right)\Bigg)\Bigg|=O_{a,A,D,\lambda}\left( \frac {x}{\log^A x}\right).
\label{equation proposition M entier}
\end{multline}
We can remove the condition of $M$ being an integer at the cost of adding the error term $O\left(x\frac{\log\log M}{M^2}\right)$. 
\end{proposition}

\begin{proof}


The proof follows closely that of Proposition 7.1 of \cite{fiorilli}. We start by splitting the sum over $q$ as follows:

$$\sum_{\substack{q\leq \frac x{rM}  \\(q,a)=1}} = \sum_{\substack{q \leq \frac x{RL}  \\(q,a)=1}}+\sum_{\substack{\frac x{RL}< q \leq \frac xr\\(q,a)=1}}-\sum_{\substack{\frac x{rM}< q \leq  \frac xr\\(q,a)=1}}. $$
We use Theorem \ref{BFI} to bound the first of these sums by taking $L:=(\log x)^{A+B+D+4}$, with $B=B(A)$ coming from this theorem:
\begin{equation*}\sum_{\substack{ R/2<r\leq R \\ (r,a)=1}}\Bigg|\sum_{\substack{q \leq \frac x{RL}  \\(q,a)=1}} \left( \psi(x;qr,a)-\Lambda(a)-\frac{x}{\phi(qr)} \right)\Bigg| \ll_{a,A,D,\lambda} \frac x{(\log x)^A}. \end{equation*}

We study the two remaining sums in the same way, by writing
$$\sum_{\substack{  \frac x{rP} <q \leq \frac x{r}  \\(q,a)=1}} \left( \psi(x;qr,a)-\Lambda(a)-\frac{x}{\phi(qr)} \right)= \sum_{\substack{\frac x{rP} <q \leq \frac x{r}  \\(q,a)=1}} \sum_{\substack{|a|<n\leq x \\ n\equiv a \bmod qr}}\Lambda(n) - x\sum_{\substack{  \frac x{rP} <q \leq \frac x{r}  \\(q,a)=1}} \frac 1{\phi(qr)}, $$
where we will take $P\leq 2L$ to be either $M$ or $\frac{RL}r$. The last term on the right is easily treated using Lemma \ref{estimation élémentaires}. As for the first term, we can remove the prime powers at the cost of a negligible error term, and end up with the following sum:
$$\sum_{\substack{ \frac x{rP} <q \leq \frac x{r} \\(q,a)=1}} \sum_{\substack{|a|<p\leq x \\ p\equiv a \bmod qr}}\log p.$$
We will now use Hooley's variant of the divisor switching technique (see \cite{hooley}). Writing $p=a+qrs$, we see that we should sum over $s$ rather than over $q$, since the bound $\frac x{rP} <q$ forces $s$ to be very small. We get that the sum is, up to an error $\ll(\log x)^2$, equal to
\begin{align*} \sum_{\substack{1\leq s < P-\frac{aP}{x} \\ (s,a)=1}} \sum_{\substack{\frac{sx}{P}+a\leq p \leq x\\p\equiv a \bmod sr}} \log p &=\sum_{\substack{1\leq s < P-\frac{aP}{x} \\ (s,a)=1}} \left( \theta(x;sr,a)-\theta\left(\frac{sx}{P}+a;sr,a\right)\right)  \\
&=\sum_{\substack{1\leq s < P-\frac{aP}{x} \\ (s,a)=1}} \frac x{\phi(sr)}\left(1-\frac sP \right) +E(r,a),
\end{align*}

where, by the Bombieri-Vinogradov theorem,
\begin{align*} \sum_{\substack{R/2<r\leq R \\ (r,a)=1}} |E(r,a)| &\leq \sum_{\substack{s\leq 2L \\ (s,a)=1}} \sum_{\substack{r\leq R \\ (r,a)=1}}  \max_{y\leq x} \left| \theta(y;sr,a)-\frac y{\phi(sr)} \right|+O_{a,A}\left(\frac x{(\log x)^A}\right)  \\
 &\leq 2L \sum_{\substack{q\leq 2RL \\ (q,a)=1}} \max_{y\leq x} \left| \theta(y;q,a)-\frac y{\phi(q)} \right|+O_{a,A}\left(\frac x{(\log x)^A}\right) \ll_A \frac x{(\log x)^A}. 
\end{align*}

Putting all this together and using the triangle inequality, we get that the left hand side of \eqref{equation proposition M entier} is
\begin{multline} \leq \sum_{\substack{R/2<r\leq R \\ (r,a)=1}} \Bigg| \sum_{\substack{s \leq \frac{RL}r \\ (s,a)=1}} \frac x{\phi(sr)}\left(1-\frac s{RL/r} \right)-\sum_{\substack{s \leq M \\ (s,a)=1}} \frac x{\phi(sr)}\left(1-\frac sM \right) - \sum_{\substack{  \frac x{RL} <q \leq \frac x{rM}  \\(q,a)=1}} \frac x{\phi(qr)} \\
- x\Bigg( \frac{C_1(a,r)}r\log (r'M)+\frac{C_3(a,r)}r -\sum_{\substack{s\leq M\\(s,a)=1}} \frac 1 {\phi(sr)} \left( 1-\frac sM\right)\Bigg)\Bigg|+O_{a,A,D,\lambda}\left(\frac x{(\log x)^A}\right),
\label{resume a borner}
 \end{multline}
since $M$ is an integer. If $M$ is not an integer, we have to add an error term of size
$$ \ll x\sum_{R/2<r\leq R} \frac{\log\log M}{\phi(r)M^2} \ll \frac{x\log\log M}{M^2}.$$ 
(We already used the fact that $x\sum_{R/2<r\leq R} \frac{\log\log (RL/r)}{\phi(r)(RL/r)^2} \ll \frac{x\log\log L}{L^2}$ in \eqref{resume a borner}.)
Applying the triangle inequality once more gives that \eqref{resume a borner} is
\begin{align*} &\leq x\sum_{\substack{R/2<r\leq R \\ (r,a)=1}} \Bigg| \sum_{\substack{s \leq \frac{RL}r \\ (s,a)=1}} \frac 1{\phi(sr)}\left(1-\frac s{RL/r} \right)-\frac{C_1(a,r)}r \log\left( \frac{r'RL}r\right) -\frac{C_3(a,r)}r \Bigg|  \\
&\hspace{1cm}+ x\sum_{\substack{R/2<r\leq R \\ (r,a)=1}}\Bigg|\sum_{\substack{  \frac x{RL} <q \leq \frac x{rM}  \\(q,a)=1}} \frac 1{\phi(qr)}- \frac{C_1(a,r)}r \log \left(\frac{RL}{rM}\right)\Bigg|+ O_{a,A,D,\lambda}\left(\frac x{(\log x)^A}\right), 
\end{align*}
which by Lemma \ref{estimation élémentaires} is
\begin{align*}
&\ll_{a,A,D,\lambda} x \sum_{\substack{R/2<r\leq R \\ (r,a)=1}} \frac{3^{\omega(r) }\log(RL)}{RL}+x \sum_{\substack{R/2<r\leq R \\ (r,a)=1}} \frac{3^{\omega(r)}\log(x/RL)}{x/RL}+\frac x{(\log x)^A} \\
& \ll \frac {x(\log R)^2}{RL}+\frac x{(\log x)^A} \\
&\ll \frac x{(\log x)^A}.
 \end{align*}

%

\end{proof}

\begin{proof}[Proof of Theorem \ref{B-V pour titchmarsh thm}]
 Taking $M=1$ in Proposition \ref{proposition M entier} and applying Lemma \ref{estimation élémentaires} and the triangle inequality, we get
$$ \sum_{\substack{\frac R2<r\leq R \\ (r,a)=1}} \left|\sum_{\substack{q\leq \frac xr \\ (q,a)=1}} (\psi(x;qr,a)-\Lambda(a))- \frac xr \left(  C_1(a,r)\log\left(\frac{(r')^2x}{er}\right) +2C_2(a,r) \right)  \right| \ll_{a,A,\lambda} \frac {x}{\log^{A+1} x}.$$
Taking dyadic intervals, one can easily use this to show that the whole sum over $r\leq R$ is $\ll_{a,A} \frac {x}{\log^{A} x} $. The result follows by exchanging the order of summation:
\begin{align*}
\sum_{\substack{q\leq \frac xr \\ (q,a)=1}} \sum_{\substack{|a|<n\leq x \\ n\equiv a \bmod qr}} \Lambda(n)  &= \sum_{\substack{|a|<n\leq x \\ n\equiv a \bmod r}} \Lambda(n) \sum_{\substack{q\leq \frac xr: \\ qr \mid n-a} } 1 \\
&=\sum_{\substack{|a|<n\leq x \\ n\equiv a \bmod r}} \Lambda(n)\tau\left( \frac{n-a}{r} \right).
\end{align*}
(the last equality is exact if $a>0$, else we have to add a neglegible error term.)

\end{proof}

\begin{proof}[Proof of Theorem \ref{main result}]
For the first result, we take $M(r,x):=M(x)$ in Proposition \ref{proposition M entier}. By Lemma \ref{somme inverse euler avec poids}, we have that 
\begin{multline}\sum_{\substack{\frac R2<r\leq R \\ (r,a)=1}} \left| \frac{\phi(a)}{a}\frac {x}{rM} \mu(a,r,M) - x\Bigg( \frac{C_1(a,r)}r\log (r'M)+\frac{C_3(a,r)}r -\sum_{\substack{s\leq M\\(s,a)=1}} \frac 1 {\phi(rs)} \left( 1-\frac sM\right)\Bigg) \right|  \\
\leq x\sum_{\substack{\frac R2<r\leq R \\ (r,a)=1}} |E(a,r,M)| \ll_a \frac x {M^{\frac{205}{538}-\epsilon}}\sum_{\substack{\frac R2<r\leq R \\ (r,a)=1}} \frac{\prod_{p\mid r}\left( 1+\frac 1{p^{\delta}}\right)}{r}  \ll \frac x {M^{\frac{205}{538}-\epsilon}},
\end{multline}
hence the result follows by the triangle inequality. 

The second result is a bit more delicate, since we have the full range of $r$, and the innermost sum depends on $R$. For this reason, we need to go back to the proof of Proposition \ref{proposition M entier}. We first split the sum over $r$ into the two intervals $r\leq R/(\log x)^{B}$ and $R/(\log x)^{B}<r \leq R$, where we take $B=B(2A)$ as in Theorem \ref{BFI}, and we can assume that $B(2A)\geq 2A$. The first part of the sum is treated using this Theorem:
\begin{multline*}  \sum_{\substack{r \leq \frac R{(\log x)^B} \\ (r,a)=1}} \left| \sum_{\substack{q\leq \frac x{RM} \\ (q,a)=1}} \left( \psi(x;qr,a)-\Lambda(a)-\frac x{\phi(qr)}\right) -\frac{\phi(a)}{a}\frac {x}{RM} \mu(a,r,M)\right|  \\
\ll_{a,A,\lambda} \frac x{(\log x)^{2A}}+\frac{x}{(\log x)^B}, 
\end{multline*}
since $\frac R{(\log x)^B} \cdot \frac x{RM} = \frac x{M(\log x)^B}\leq \frac x{(\log x)^{B}}$. 
%
For the rest of the sum, we argue as in the proof of Proposition \ref{proposition M entier}. We split the sum over $q$ as follows:
$$\sum_{\substack{q\leq \frac x{RM}  \\(q,a)=1}} = \sum_{\substack{q \leq \frac x{RL}  \\(q,a)=1}}+\sum_{\substack{\frac x{RL}< q \leq \frac xr\\(q,a)=1}}-\sum_{\substack{\frac x{RM}< q \leq  \frac xr\\(q,a)=1}}. $$
Taking $P$ to be either $\frac RrL$ or $\frac Rr M$, we have that $P\leq L (\log x)^B$ (instead of $P\leq 2L$). The rest of the proof goes through, and we get that 
\begin{multline}
\sum_{\substack{\frac RL<r\leq R \\ (r,a)=1}}\Bigg|\sum_{\substack{q\leq \frac x {RM}  \\(q,a)=1}} \left( \psi(x;qr,a)-\Lambda(a)-\frac{x}{\phi(qr)} \right) - x\Bigg( \frac{C_1(a,r)}r\log \left(r'RM/r\right)+\frac{C_3(a,r)}r  \\
-\sum_{\substack{s\leq RM/r\\(s,a)=1}} \frac 1 {\phi(rs)} \left( 1-\frac s{RM/r}\right)\Bigg)\Bigg|\ll_{a,A,D,\lambda} \frac {x}{(\log x)^{2A}}+E_2(x,M),
\label{equation proposition M entier pour R}
\end{multline}
where $E_2(x,M)$ is the error coming from the fact that $\frac Rr M$ is not an integer, which is
$$ \ll x\sum_{\frac RL <r\leq R}\frac {\log\log (RM/r)}{\phi(r) RM/r} \frac 1{RM/r} \ll  \frac {x}{(RM)^2}\sum_{\frac RL <r\leq R}\frac {r^2\log\log (RM/r)}{\phi(r)} \ll \frac{x \log\log M}{M^2}. $$
We finish the proof by applying Lemma \ref{somme inverse euler avec poids} and the triangle inequality.
\end{proof}

\begin{proof}[Proof of Corollary \ref{corollaire limitations}]
By the triangle inequality we have
\begin{multline*}
 \sum_{\substack{r\leq R\\ (r,a)=1}} \left|\frac{\phi(a)}a\frac{x}{RM}\mu(a,r,RM/r) \right|\leq  \sum_{\substack{r\leq R \\ (r,a)=1}} \Bigg|\sum_{\substack{q\leq \frac x{RM} \\ (q,a)=1}} \left( \psi(x;qr,a)-\Lambda(a)-\frac x{\phi(qr)}\right) \\
 -\frac{\phi(a)}a\frac{x}{RM}\mu(a,r,RM/r) \Bigg|
 +\sum_{\substack{r\leq R \\ (r,a)=1}} \left|\sum_{\substack{q\leq \frac x{RM} \\ (q,a)=1}} \left( \psi(x;qr,a)-\Lambda(a)-\frac x{\phi(qr)}\right) \right|,
\end{multline*}
hence by Theorem \ref{main result} we get the lower bound
\begin{multline*} \sum_{\substack{r\leq R \\ (r,a)=1}} \left|\sum_{\substack{q\leq \frac x{RM} \\ (q,a)=1}} \left( \psi(x;qr,a)-\Lambda(a)-\frac x{\phi(qr)}\right) \right| \geq \frac{\phi(a)}a\frac{x}{RM}\sum_{\substack{r\leq R \\ (r,a)=1}} |\mu(a,r,RM/r)| \\
 - O_{\epsilon}\left(\frac {x}{M^{\frac{743}{538}-\epsilon}} \right),
\end{multline*}

since for $M$ large enough, $\mu(a,r,RM/r)\leq 0$. For the upper bound, we write
\begin{align*} \sum_{\substack{r\leq R \\ (r,a)=1}} &\left|\sum_{\substack{q\leq \frac x{RM} \\ (q,a)=1}} \left( \psi(x;qr,a)-\Lambda(a)-\frac x{\phi(qr)}\right) \right| \leq \sum_{\substack{r\leq R \\ (r,a)=1}} \Bigg|\sum_{\substack{q\leq \frac x{RM} \\ (q,a)=1}} \left( \psi(x;qr,a)-\Lambda(a)-\frac x{\phi(qr)}\right)  \\
 &\hspace{3cm}-\sum_{\substack{r\leq R \\ (r,a)=1}}\frac{\phi(a)}a\frac{x}{RM}\mu(a,r,RM/r)\Bigg| 
+\sum_{\substack{r\leq R \\ (r,a)=1}} \left|\frac{\phi(a)}a\frac{x}{RM}\mu(a,r,M)\right|  \\
&\hspace{3cm}\leq \frac{\phi(a)}a\frac{x}{RM}\sum_{\substack{r\leq R \\ (r,a)=1}} |\mu(a,r,RM/r)| + O_{\epsilon}\left(\frac {x}{M^{\frac{743}{538}-\epsilon}} \right).
%
%
\end{align*}
 The result follows by the definition of $\mu(a,r,RM/r)$. Note that if $a=\pm 1$, then we have
\begin{align*}
 2\sum_{\substack{r\leq R \\ (r,a)=1}} |\mu(a,r,RM/r)| &= \sum_{\substack{r\leq R}} \left(\log (RM/r)+2C_5 + \sum_{p\mid r}\frac{\log p}p\right) 
 \\ &= (R+O(1))\left(\log M+1+2C_5 +O\left(\frac{\log R}R\right)\right) + \sum_{p\leq R} \frac{\log p}p \left \lfloor \frac Rp\right \rfloor,
\end{align*}
by Stirling's approximation. The last sum can be handled without much effort:
\begin{align*}\sum_{p\leq R} \frac{\log p}p \left \lfloor \frac Rp\right \rfloor &= R \sum_{p\leq R} \frac{\log p}{p^2}+ O\left(\sum_{p\leq R} \frac{\log p}p\right) \\
&= R\left(\sum_{p}\frac{\log p}{p^2}+O\left(\frac 1R \right)\right) + O\left(\log R \right). 
\end{align*}
Hence,
$$ \sum_{\substack{r\leq R \\ (r,a)=1}} |\mu(a,r,RM/r)| =R\left(\frac 12\log M+C_6\right)+O(\log (RM)). $$

\end{proof}

\begin{proof}[Proof of Proposition \ref{proposition appliquer titchmarsh a BFI}]
 Exchanging the order of summation as in the proof of Theorem \ref{B-V pour titchmarsh thm}, we get that
$$\sum_{\substack{\frac R2<r\leq R \\ (r,a)=1}} \left |\sum_{\substack{q\leq \frac xr \\ (q,a)=1}} \left( \psi(x;qr,a)-\Lambda(a)-\frac x{\phi(qr)} \right)- x\Bigg( \frac{C_1(a,r)}r\log r'+\frac{C_3(a,r)}r\Bigg) \right| \ll \frac x{(\log x)^A}.$$ 
As we have seen in the proof of Proposition \ref{proposition M entier}, we can 
give a good estimate for the part of the sum over $q$ where $\frac x{RL}<q\leq \frac x{r}$ by switching divisors and using the Bombieri-Vinogradov theorem (which explains the restriction on $R$). Doing so and applying Lemma \ref{somme inverse euler avec poids}, we get that
$$\sum_{\substack{\frac R2<r\leq R \\ (r,a)=1}} \left |\sum_{\substack{q\leq \frac x{RL} \\ (q,a)=1}} \left( \psi(x;qr,a)-\Lambda(a)-\frac x{\phi(qr)} \right) \right| \ll \frac {x}{L},$$ 
which concludes the proof.
\end{proof}


\begin{thebibliography}{99}


\bibitem{BFI} Enrico Bombieri, John B. Friedlander and Henryk Iwaniec, \emph{Primes in arithmetic progressions to large moduli.} Acta Math. \textbf{156}  (1986),  no. 3-4, 203--251.

\bibitem{BFI2} Enrico Bombieri, John B. Friedlander and Henryk Iwaniec, \emph{Primes in arithmetic progressions to large moduli. II.} Math. Ann. \textbf{277} (1987), no. 3, 361--393.

\bibitem{BFI3} Enrico Bombieri, John B. Friedlander and Henryk Iwaniec, \emph{Primes in arithmetic progressions to large moduli. III.} J. Amer. Math. Soc. \textbf{2} (1989), no. 2, 215--224.

\bibitem{felix} Adam Tyler Felix, \emph{Generalizing the Titchmarsh divisor problem.} preprint.

\bibitem{fouvry BV} Étienne Fouvry, \emph{Autour du théorème de Bombieri-Vinogradov.} Acta Math. \textbf{152} (1984), no. 3-4, 219--244.

\bibitem{fouvry BV2} Étienne Fouvry, \emph{Autour du théorème de Bombieri-Vinogradov. II.} Ann. Sci. École Norm. Sup. (4) \textbf{20} (1987), no. 4, 617--640.

\bibitem{fouvry titchmarsh} Étienne Fouvry, \emph{Sur le problème des diviseurs de Titchmarsh.} J. Reine Angew. Math. \textbf{357} (1985), 51--76.

\bibitem{fouvry iwaniec BV} Étienne Fouvry, Henryk Iwaniec, \emph{On a theorem of Bombieri-Vinogradov type.} Mathematika \textbf{27} (1980), no. 2, 135--152 (1981). 

\bibitem{fouvry iwaniec PAP} Étienne Fouvry, Henryk Iwaniec, \emph{Primes in arithmetic progressions.} Acta Arith. \textbf{42} (1983), no. 2, 197--218.

\bibitem{fiorilli} Daniel Fiorilli, \emph{Residue classes containing an unexpected number of primes.} preprint: arXiv:1009.2699v1  [math.NT].


\bibitem{FGHM} John B. Friedlander, Andrew Granville, Adolf Hildebrand, Helmut Maier, \emph{Oscillation theorems for primes in arithmetic progressions and for sifting functions.} J. Amer. Math. Soc. \textbf{4} (1991),  no. 1, 25--86.

\bibitem{GPY} Daniel A. Goldston, J\'anos Pintz, Cem Y. Yildirim, \emph{Primes in tuples. I.} Ann. of Math. (2) \textbf{170} (2009), no. 2, 819--862.


\bibitem{hooley} Christopher Hooley, \emph{On the Barban-Davenport-Halberstam theorem. I.} Collection of articles dedicated to Helmut Hasse on his seventy-fifth birthday, III. J. Reine Angew. Math. \textbf{274/275} (1975), 206--223
 
\bibitem{huxley} Martin N. Huxley, \emph{Exponential sums and the Riemann zeta function. V.} Proc. London Math. Soc. (3) \textbf{90} (2005), no. 1, 1--41. 





\end{thebibliography}
\end{document}